\documentclass[10pt,a4paper]{article}
\usepackage{amsmath}
\usepackage{amsthm}
\usepackage{amsfonts}
\usepackage{amssymb}
\usepackage{graphicx}
\usepackage{dcolumn}
\usepackage{verbatim}
\usepackage{latexsym}
\usepackage{fancyhdr}
\newcommand{\version}{8f628b787780}
\usepackage[
    pdftex,
    a4paper,
    bookmarks,
    bookmarksopen=False,
    pdfauthor={Mathias Carlen, Henryk Gerlach},
    pdftitle={A closed contact cycle on the ideal trefoil \version},
]{hyperref}

\newcommand{\scriptL}{\mathcal{L}} 
\newcommand{\trefoil}{\gamma_{3_1}}

\newtheorem{definition}{Definition}
\newtheorem{conjecture}{Conjecture}

\newtheorem{remark}{Remark}
\newtheorem{lemma}{Lemma}
\newtheorem{proposition}{Proposition}

\newcommand{\NN}{\mathbb{N}}
\newcommand{\ZZ}{\mathbb{Z}}
\newcommand{\RR}{\mathbb{R}}
\newcommand{\Sph}{\mathbb{S}}
\newcommand{\diam}{{\rm diam}}

\newcommand{\pp}{{\rm pp}}

\newlength\smallpic
\begin{document}

\title{A closed contact cycle on the ideal trefoil}
\date{\today}

\author{M.~Carlen, H.~Gerlach}
\maketitle
\begin{center}
{Institut de Math\'{e}matiques B, \'{E}cole Polytechnique 
F\'ed\'erale de Lausanne, CH-1015 Lausanne, Switzerland, 
\{mathias.carlen, henryk.gerlach\}@gmail.com}
\end{center}

\begin{abstract}
Numerical computations suggest that each point on a certain optimized shape called the ideal trefoil is in contact with two other points.
We consider sequences of such contact points, such that each point is in contact with its predecessor and call it a billiard.
Our numerics suggest that a particular billiard on the ideal trefoil closes to a periodic cycle after nine steps.
This cycle also seems to be an attractor: all billiards converge to it.
\end{abstract}

\section{Introduction}
A closed curve in $\RR^3$ is called ideal if it minimizes its ropelength $\scriptL[\cdot]/\Delta[\cdot]$
-- i.e. its length divided by its thickness -- within its knot class \cite{GM99}. 
In this paper we will focus on the simplest of all proper ideal knots, 
namely the trefoil knot. Various numerical approximations of this specific
knot are available. It is not trivial to define what the properties of a  ``good'' approximation are.
Quantities like ropelength, functions such as curvature and torsion, or the contact set for a given knot are
can all be used to assess whether a knot is close to ideal. There exist
several algorithms to compute ideal knot shapes, which use
different approximations for the curve description\cite{L98,bookchapter,P98,CPR05,ACPR10}. These
numerical computations are expected to lead to a better
understanding of ideal knots. In this sense a numerical shape is ``good'',
if it leads to more insight about properties of ideal knots. 

A curve is in contact with itself at the points $p,q$ if the distance between $p$ and $q$
is precisely two times the thickness of the curve and the line segment between them is orthogonal
to the curve at both ends. \cite{S04, bookchapter} define a robust sense of contact with a tolerance and their computations 
suggest that each point on the ideal trefoil in $\RR^3$ is in contact with two other points . 
Starting
from a point $p_0$, it is in contact with a point $p_1$ that itself is again in contact with a point $p_2\neq p_0$ and
so on. Does this sequence close to a cycle?
In this article we observe that computations suggest that the ideal trefoil
knot has a periodic, and attracting nine-cycle of contact chords,
as illustrated in Figure \ref{fig:billiard}.

A similar construction of periodic cycles, but in each point of the curve, 
helped to construct the ideal Borromean rings \cite{Sta03,CFKSW06}.
The existence of this cycle is significant because it partitions the trefoil in such a way
that, using the apparent symmetries, it can be re-constructed from two unknown small pieces of 
curves mutually in contact.

\begin{figure}
\begin{center}
\includegraphics[width=.6\textwidth]{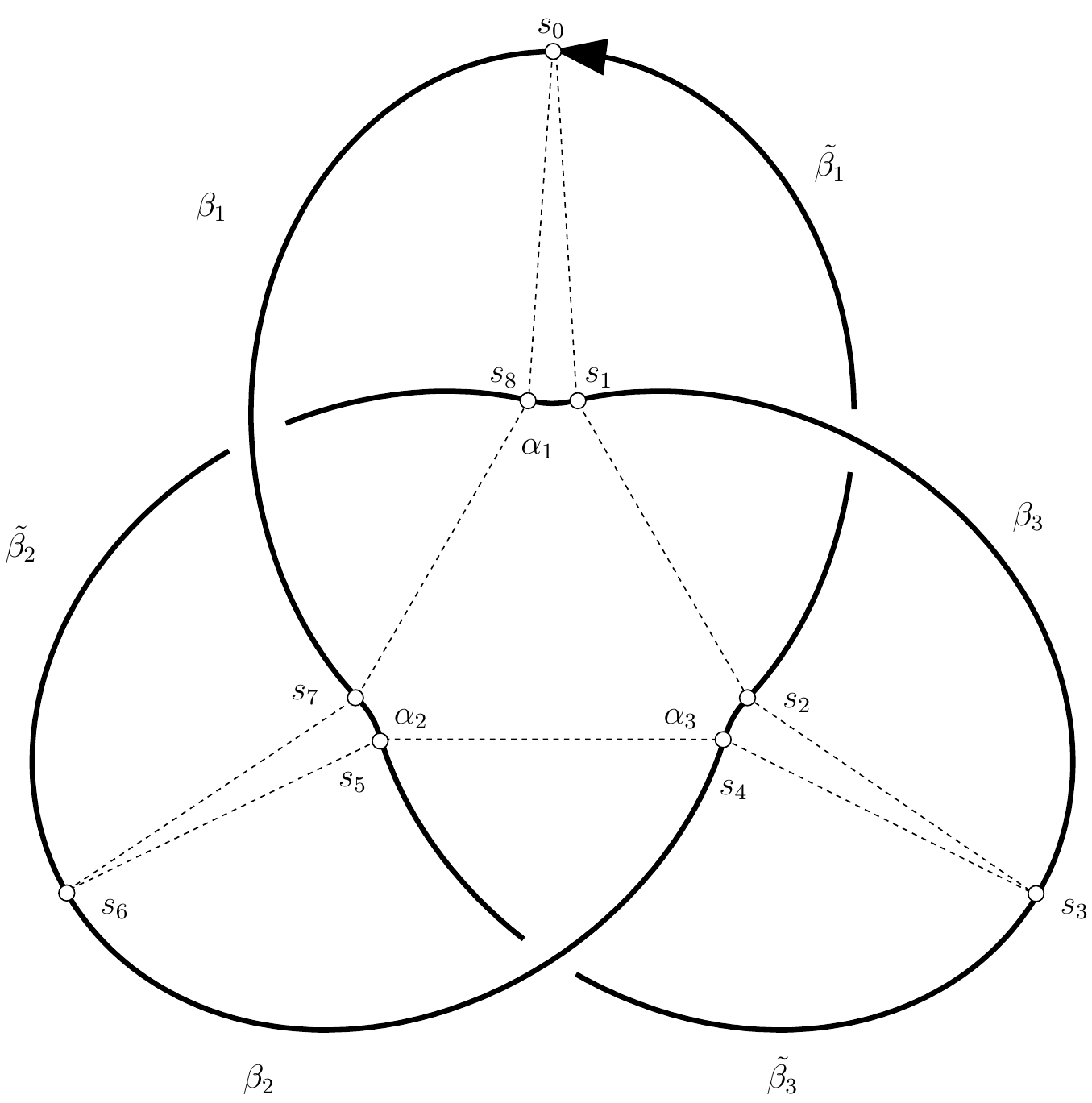} 
\caption{ The parameters $s_i := \sigma^i(0)$ of the nine-cycle $b_9$ partition the trefoil in 9 curves: 
          $\beta_i$ and $\tilde{\beta_i}$ are all congruent as are $\alpha_i$ ($i=1,2,3$). The contact function 
          $\sigma$ maps the parameter 
          interval of each curve bijectively to the parameter interval of another curve (see Figure 
          \ref{fig:connect}).
        } \label{fig:billiard}
\end{center}
\end{figure}

We approximated the ideal trefoil using a Fourier representation described in \cite{CG10}.
The numerical computations were carried out with \texttt{libbiarc} \cite{LB10}
and the data is available from \cite{Data10}.
The numerical Fourier trefoil is not the best known in ropelength sense, but -- to our knowledge -- 
the best shape to observe the closed cycle, probably because we can enforce specific
symmetries. 

Another interesting discovery is that, if we follow the contact chords starting
at an arbitrary point on the trefoil, we always end up at
the previously mentioned cycle, in other words, it is a global attractor.

In order to present this closed cycle on the trefoil we first
review the notions of global radius of curvature \cite{GM99}
and contact of a curve \cite{PP02, bookchapter} in Section \ref{sec:grc}. 
Section \ref{sec:billiards} introduces contact billiards and cycles. Then
we present and discuss a candidate for a cycle in the trefoil
and show numerically that it seems to act as an attractor for all the billiards
on the trefoil.

\section{The Ideal Trefoil -- Its Contact Chords and Symmetries}\label{sec:grc}
A knot is a closed curve $\gamma \in C^1(\Sph,\RR^3)$ 
where $\Sph:=\RR/\ZZ$ is the unit interval with the endpoints identified, isomorphic to the unit circle.
We use the global radius of curvature to assign a thickness $\Delta$
to $\gamma$.
\begin{definition}[Global radius of curvature] \cite{GM99} \label{def:grc}
    For a $C^0$-curve  $\gamma: \Sph \longrightarrow \RR^3$ the global radius of
    curvature at $s \in \Sph$ is
    \begin{equation}
        \rho_G[\gamma](s) := \inf_{\sigma,\tau \in \Sph, \sigma \not= \tau, \sigma\not=s, \tau\not=s}
                R(\gamma(s),\gamma(\sigma), \gamma(\tau)).
    \end{equation}
    Here $R(x,y,z) \ge 0$ is the radius of the smallest circle
    through the points $x,y,z \in \RR^3$,  i.e.
    \[
        R(x,y,z) :=
        \left\{
        \begin{array}{cl}
                \frac{|x-z|}{2 \sin \measuredangle(x-y,y-z)} &
                \mbox{$x,y,z$ not collinear}, \\
        \infty &
                \mbox{$x,y,z$ collinear, pairwise distinct}, \\
        \frac{\diam(\{x,y,z\})}{2} & \mbox{otherwise}.
        \end {array}
        \right.
     \]
    where $\measuredangle(x-y,y-z) \in [0,\pi/2]$ is the smaller angle between
    the vectors $(x-y)$ and $(y-z) \in \RR^3$, and
    \[
        \diam(M):= \sup_{x,y \in M} |x-y| \; \mbox{for} \; M \subset \RR^3
    \]
    is the diameter of the set $M$.
    The {\it thickness of $\gamma$}, denoted as
    \begin{equation} \label{1.1} \Delta[\gamma] := \inf_{s \in \Sph} \rho_G[\gamma](s)
                     = \inf_{s,\sigma,\tau \in \Sph, \sigma \not= \tau,\sigma\not=s,\tau\not=s}
                       R(\gamma(s),\gamma(\sigma), \gamma(\tau)),
    \end{equation}
    is defined as the infimum of $\rho_G$.
\end{definition}
A curve that minimizes arclength over thickness is called an ideal
knot \cite{KBMSDS96,GM99}.
Already \cite{GM99} showed in a $C^2$-setting that for a knot to be ideal,
$\rho_G$ around a parameter\footnote{The proof from \cite{GM99} only requires the curve to be $C^2$ on a neighborhood of the parameter, not everywhere.} 
is either constant and equal to
the infimum, or the curve is locally a straight line. 
A proof of this necessary condition for $C^{1,1}$ curves 
is not known yet. 
Assume for a moment\footnote{So far, the numerical shapes suggest that most ideal knots are $C^2$ except for a finite number of points.}, 
that $\gamma$ is ideal and $C^2$, then for each $t \in \Sph$ we distinguish
the following three cases \cite{GM99,LSDR99}:

\begin{enumerate}
  \item[(A)] $\rho_G(t) > \Delta[\gamma]$ and there 
             exists $\varepsilon>0$ such that $\gamma(\{s: |s-t|<\varepsilon\})$ is a straight line.
  \item[(B)] $\rho_G(t) = \Delta[\gamma]$ and the curvature of $\gamma$ at $t$ is $1/\Delta[\gamma]$.
  \item[(C)] $\rho_G(t) = \Delta[\gamma]$ and 
     there exists a $s \in \Sph$ with $|\gamma(s)-\gamma(t)|=2\Delta[\gamma]$ and 
     $\langle\gamma'(s), \gamma(s)-\gamma(t)\rangle = \langle\gamma'(t), \gamma(s)-\gamma(t)\rangle=0$.
\end{enumerate}
In case (B) we say that global curvature is attained locally, or that curvature is active, while in case (C) we say
that the contact is global.\footnote{For $C^{1,1}$-curves the situation is less clear but the cases (B) and (C) remain interesting.} 
The global contact is realized by a contact chord.

\begin{definition}[Contact Chord]\label{def:contact_chord}
  Let $\gamma \in C^1(\Sph,\RR^3)$ be a regular, i.e. $|\gamma'(s)|>0$, curve with $\Delta[\gamma]>0$ and
  let $s,t \in \Sph$ be such that $c(s,t) := \gamma(t)-\gamma(s)$ has length
  \[ 
    |c(s,t)|=2\Delta[\gamma],
  \]
  and $c(s,t)$ is orthogonal to $\gamma$, i.e.  
     $\langle\gamma'(s), c(s,t))\rangle = \langle\gamma'(t), c(s,t)\rangle=0,$
  then we call $c(s,t)$ a {\it contact chord}.
  If such $s$ and $t$ exist, we say $\gamma$ has a
  contact chord connecting $\gamma(s)$ and $\gamma(t)$ or the parameters
  $s$ and $t$ are (globally) in contact.
  The set
  \[
     \{\gamma(s)+h c(s,t): h \in [0,1]\} \subset \RR^3
  \]
  will also be called a contact chord. Being in contact is a symmetric
  relation.
\end{definition}

\begin{figure}
  \begin{center}
  \begin{tabular}{|ll|}
      \hline
      Name  & \verb:k3_1:\\
      Degrees of freedom & 165 \\
      Biarc nodes & 333 \\
      Arclength $\scriptL$ & 1 \\
      Thickness $\Delta$ & 0.030539753 \\
      Ropelength $\scriptL/\Delta$ & 32.744208 \\
      \hline 
    \end{tabular} \\
    \vspace{.4cm}
    \begin{tabular}{ccc}
    \includegraphics[width=.3\textwidth]{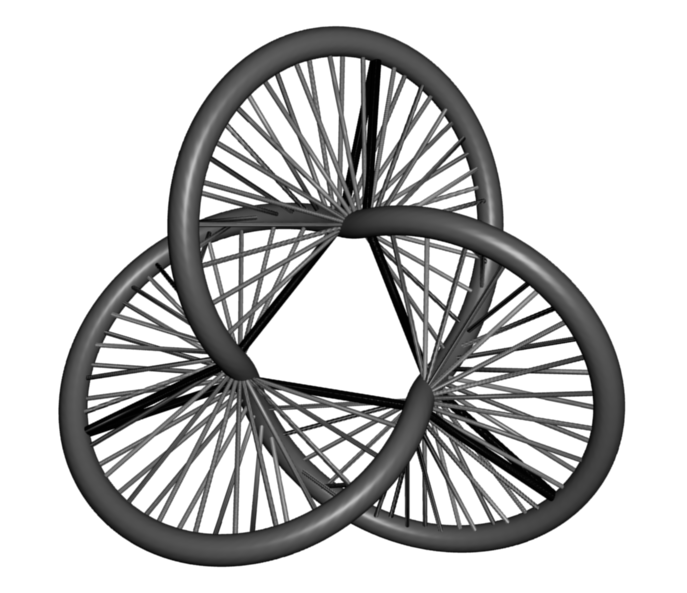} &
    \includegraphics[width=.3\textwidth]{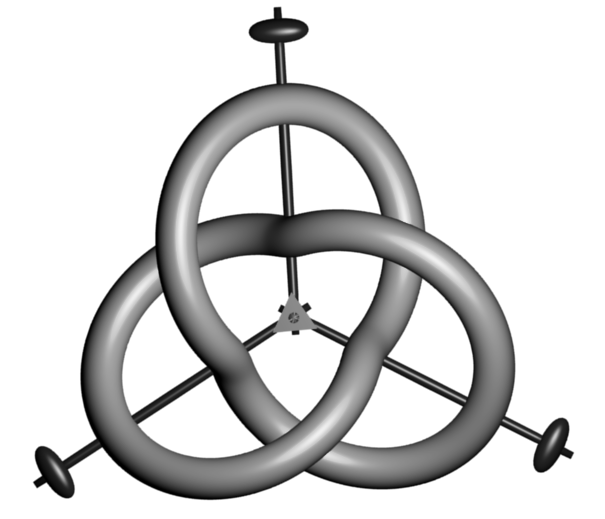} &
    \includegraphics[width=.3\textwidth]{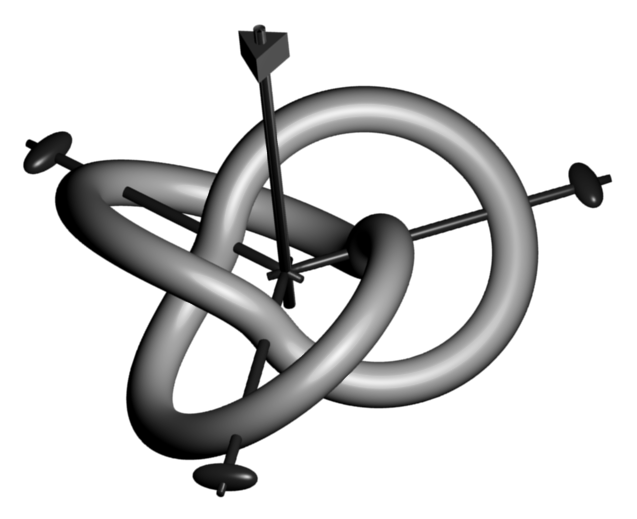} \\
    (a) & (b) & (c)
    \end{tabular}
  \end{center}
  \caption{In the top box we describe the data for the trefoil used
           for the computations.
           In the bottom row we have (a) the trefoil with a
           few contact chords. The dark chords visualize the closed cycle
           in the contact chords of the trefoil. The pictures (b) and
           (c) illustrate two different views of plausible symmetry axes.
           The $120^\circ$ rotation axis and the three $180^\circ$ rotation axes are decorated with a prism and an ellipsoid at
           the end, respectively.
           \label{fig:billiard_struts} }
\end{figure}

For the rest of the article, we will restrict ourselves to the ideal trefoil\footnote{It is widely assumed that the ideal
trefoil is unique but it remains to be proven rigorously.} $\trefoil$.
The trefoil data used in this article was computed as in \cite{CG10}.
A Fourier representation of the knot makes enforcing symmetries
natural. The specific symmetries are proposed in Conjecture \ref{con:sym}.
They significantly reduce the number of independent Fourier parameters in
simulated annealing \cite{L98}, while the computation
of the thickness is done by interpolating the Fourier knot
with biarcs \cite{bookchapter,CG10}. The Fourier coefficients
and the point-tangent data for the trefoil are available online
\footnote{Data is available at \cite{Data10}.
\texttt{k3\_1.3} with MD5 sum \texttt{cf5e2f8550c4c1e91a2fd7f5e9830343}
and \texttt{k3\_1.pkf} with sum \texttt{531492b73b2ec4be2829f6ab2239d4d5}.
} (see also Figure \ref{fig:billiard_struts}).

The numeric approximations of the ideal trefoil suggest that every point is globally in contact
with two other points on the trefoil \cite{bookchapter}. 
We can sort the contact chords in a continuous fashion, such that each point has an incoming and an outgoing
contact. 
In our numerical computations of the contact chords we used the point-to-point
distance function
\[
\pp(s,\sigma):=|\gamma(s)-\gamma(\sigma)|.
\]
The general belief is that the $\pp$-function of the ideal trefoil has an extremely flat double valley away from the diagonal 
\cite{bookchapter} (see Figure \ref{fig:pp3d} for a 3D version of $\pp(s,\sigma)$ for the trefoil).
For a sampling $s_i:=i/n$, $i=0,\ldots,n$, we compute $\sigma_i$ as the minimum of $\pp(s_i,\cdot)$
restricted to the region $[\sigma_{i-1}-\varepsilon,\sigma_{i-1}+\varepsilon]$,
$1\gg\varepsilon>0$. 
The initial
value $(s_0,\sigma_0)$ is computed as the local minimum away from the diagonal.
We now choose one of the two valley floors. By staying close to the previously computed minimum,
we never cross over to the second valley. We then 
linearly interpolate between the $(s_i,\sigma_i)$ pairs to obtain an approximation of the so called contact
function $\sigma$ (also see Figure \ref{fig:connect} below for a top view of $\sigma$):

\begin{figure} 
  \begin{center} 
      \includegraphics[width=.8\textwidth]{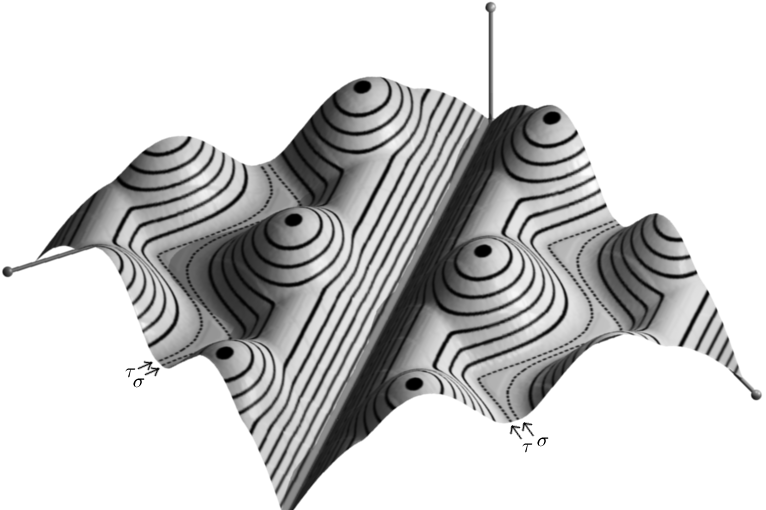} 
  \end{center} 
 \caption{The distance function $\pp(s,\sigma)$ for $s,\sigma\in \Sph$ goes to $0$ on the diagonal and forms a large valley with two very shallow sub-valleys. 
 The dotted lines marked with arrows in the valley indicate the two local minima, i.e. $\sigma(s)$ and 
 $\tau(s)$.} 
 \label{fig:pp3d} 
\end{figure}

\begin{definition}[Contact functions]\label{def:sigma}
  Let $\sigma:\Sph\longrightarrow\Sph$ be a continuous, bijective and orientation preserving function, such that
  $\trefoil(s)-\trefoil(\sigma(s))$ is a contact chord for every $s \in \Sph$.
  The inverse function of $\sigma$ is $\tau:=\sigma^{-1}$.
\end{definition}

As mentioned previously, numerics point out that the trefoil is
symmetric with respect to a specific symmetry group \cite{bookchapter,BPP08,CG10}.
These symmetries have helped to identify the closed cycle proposed
later in this article.

\begin{conjecture}[Symmetry of the ideal trefoil]\label{con:sym}
The ideal trefoil has symmetries as shown in Figure \ref{fig:billiard_struts}.
\end{conjecture}

The symmetries of the trefoil are also apparent in its contact functions $\sigma$
and $\tau$. The relations are listed in the following lemma.

\begin{lemma}[Symmetry of $\sigma,\tau$]\label{lem:sigma_sym}
  Assume Conjecture \ref{con:sym} about the symmetry of the constant speed parameterized trefoil 
  $\trefoil: \Sph \longrightarrow \RR^3$ is true. Then the contact 
  functions $\sigma,\tau: \Sph\longrightarrow \Sph$ have the following properties:
  \begin{eqnarray}
    \sigma(s^*+t) &\stackrel{\text{in\;} \Sph}{=}& -\tau(s^*-t), \quad  \forall t\in \Sph \\
    \sigma(t+1/3) &\stackrel{\text{in\;} \Sph}{=}& \sigma(t)+1/3,\quad  \forall t\in \Sph \\
    \tau(s^*+t) &\stackrel{\text{in\;} \Sph}{=}& -\sigma(s^*-t), \quad  \forall t\in \Sph \\
    \tau(t+1/3) &\stackrel{\text{in\;} \Sph}{=}& \tau(t)+1/3,    \quad  \forall t\in \Sph
   \end{eqnarray}
  where $s^*\in\Sph$ is a parameter such that $\trefoil(s^*)$ is on a $180^\circ$ rotation axis.  \qed
\end{lemma}

\section{Closed Cycles} \label{sec:billiards}

Recall from the previous section that numerics suggest that every point on the ideal trefoil $\trefoil$
is in contact with two other points and we assume to be able to define a contact 
function $\sigma$ as in Definition \ref{def:sigma}.
Is there a finite tuple of points such that each parameter is in contact with, and only with, its cyclic 
predecessor and successor?
Inspired by {\it Dynamical Systems} \cite{B27} we call a sequence of parameters
that are in contact with each predecessor a billiard.
If a billiard closes, we call it a cycle:
\begin{definition}[Cycle]
  For $n \in \NN$ let $b:=(t_0,\ldots,t_{n-1}) \in \Sph\times\cdots\times\Sph$ be an $n$-tuple.
  We call $b$ an $n$-cycle if $\sigma(t_i)=t_{i+1}$ for $i=0,\ldots,n-2$ and $\sigma(t_{n-1})=t_0$, where
  $\sigma$ is defined as in Definition \ref{def:sigma}.
  The cycle $b$ is called minimal if all $t_i$ are pairwise distinct.
\end{definition}

Each cyclic permutation of a cycle is again a cycle. Basing the definition of cycles on the
continuous function $\sigma$ instead of closed polygons in $\RR^3$ makes it slightly easier to find them numerically:
\begin{remark}
  The $\trefoil$ has an $n$-cycle iff there exists some $t \in \Sph$  such that
  \[
    \sigma^n(t):=\underbrace{\sigma\circ\cdots\circ\sigma}_{n \text{ times }}(t)=t.
  \]
  The cycle is then $b:=(t,\sigma^1(t),\ldots,\sigma^{n-1}(t))$. \hfill$\Box$
\end{remark}

All parameters of a minimal $n$-cycle are pairwise distinct so each minimal $n$-cycle corresponds
to $n$ points in the set $\{t \in \Sph: \sigma^n(t) = t\}$. Since there are $n$ cyclic permutations of an $n$-cycle
and since minimal $n$-cycles that are not cyclic permutations must be point-wise distinct this leads to:
\begin{lemma}[Counting Cycles]\label{lem:counting}
  Define the set of intersections of $\sigma^n$ with the diagonal $I := \{t \in \Sph: \sigma^n(t) = t\}$.
  If there is a finite number of minimal $n$-cycles then 
  \begin{eqnarray*}
   \#I &\ge& (\text{Number of distinct minimal $n$-cycles})\cdot n  \\
       & = & (\text{Number of minimal $n$-cycles}).
  \end{eqnarray*}
  \qed
\end{lemma}

In Figure \ref{fig:sigmatable} we compiled small plots of $\sigma^n$  for $n=1,\ldots,27,144$. 
For $n=2$ the
function $\sigma^2$ comes close to the diagonal for the first time, but can not touch it in less than three points
by Lemma \ref{lem:sigma_sym}. If it would touch it would have to touch at least six times by 
Lemma \ref{lem:counting}, which does not seem to be the case.\par

By similar arguments, we exclude the possibility of cycles for $n=7,$ $11,$ $13,$ $16,$ $20$ and $25$. 
On the other hand the
case $n=9$ looks promising (see Figure \ref{fig:sigma9}). It seems to touch the diagonal precisely nine times which
suggests the existence of a single minimal cycle $b_9$ and its cyclic permutations. With our parameterization 
the cycle $b_9$ happens to start at $0$ and we compute a numerical error of only $\sigma^9(0) = 0.0007 \approx 0$.
Consequently the cases $n=18,27$ would also touch the diagonal, but the corresponding cycle would not be minimal.
We studied the plots till $n=100$, but did not find any other promising candidates (apart from $n=k\cdot 9$ for $k\in\NN$). 
Keep in mind that the numerical error increases with $n$, but even for $n=144$ the graph looks reasonable. 

\begin{figure}
\begin{center}
\begin{tabular}{ccccc}
\includegraphics[width=\smallpic]{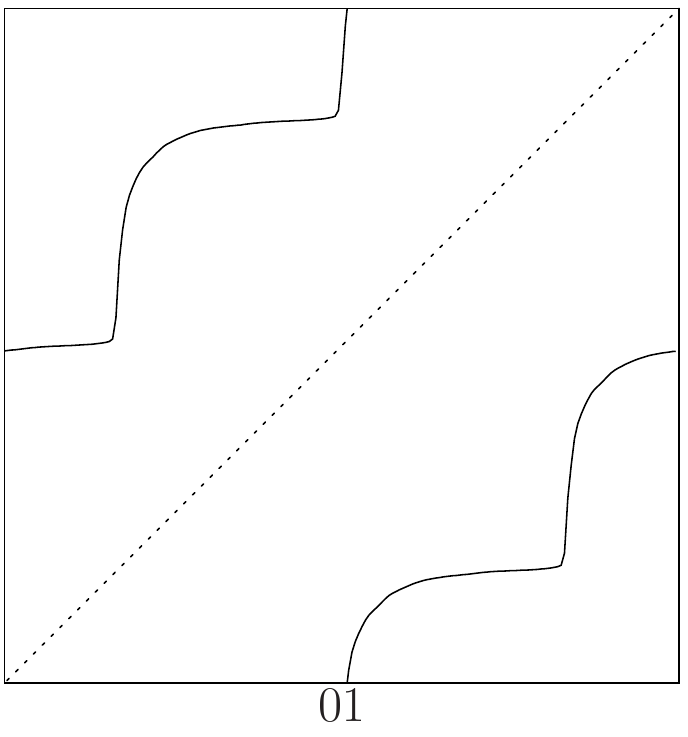} &
\includegraphics[width=\smallpic]{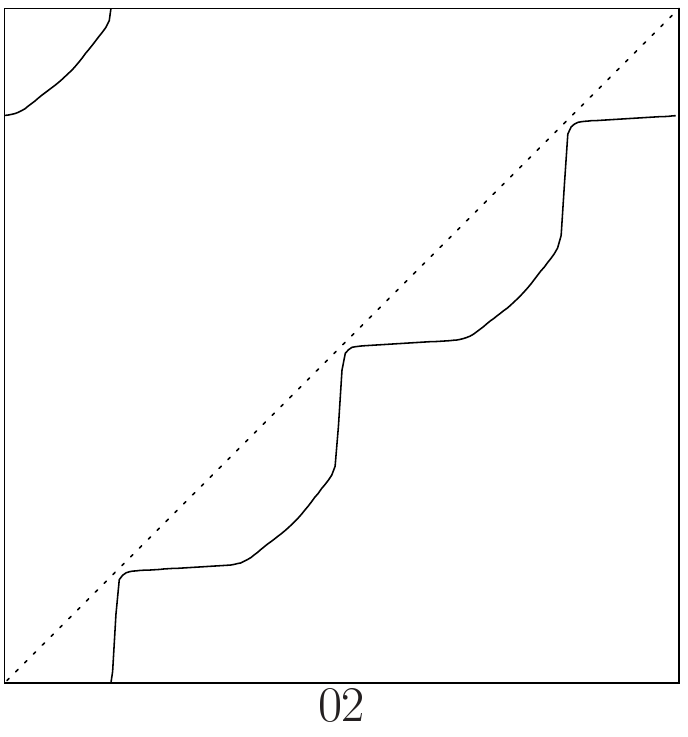} &
\includegraphics[width=\smallpic]{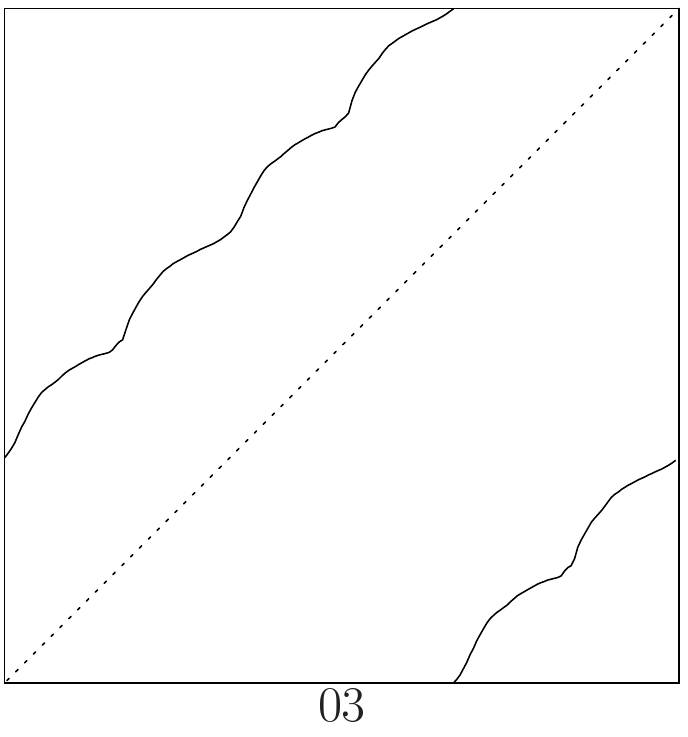} &
\includegraphics[width=\smallpic]{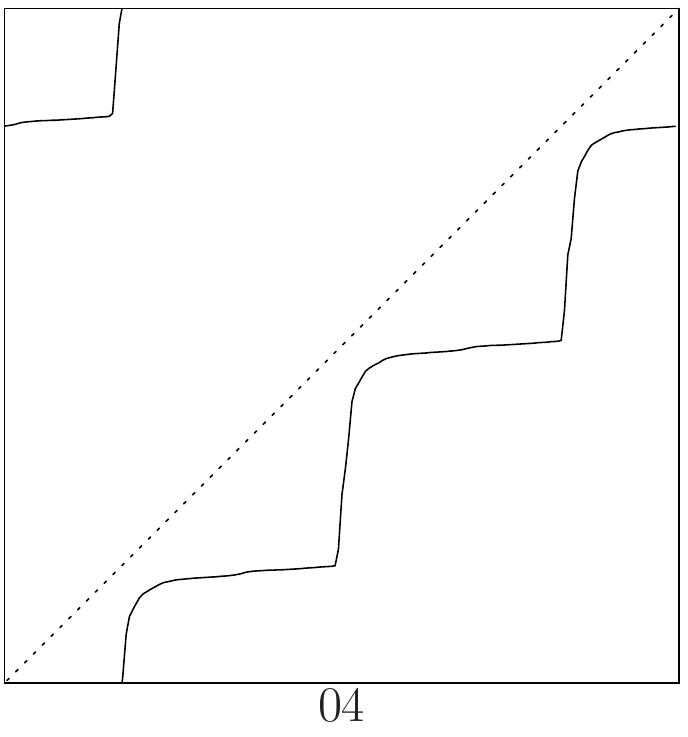} \\
\includegraphics[width=\smallpic]{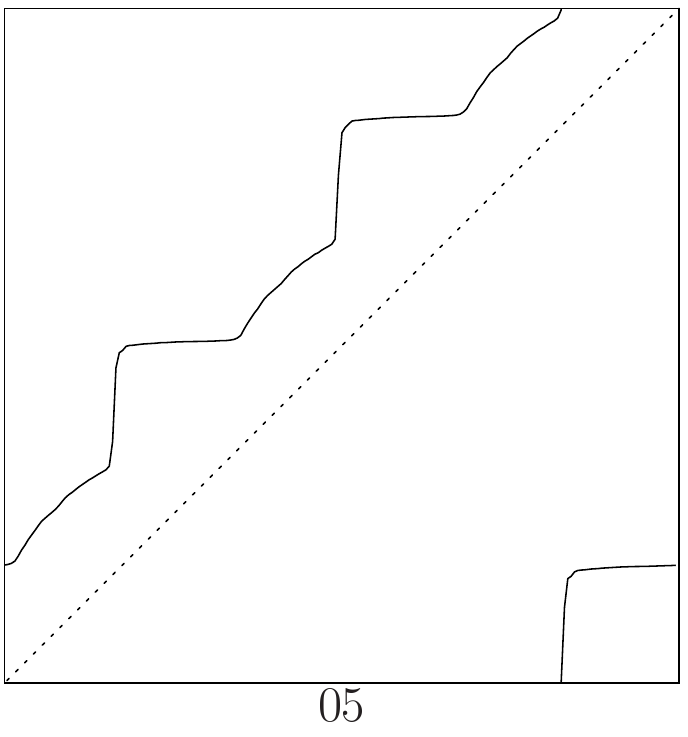} &
\includegraphics[width=\smallpic]{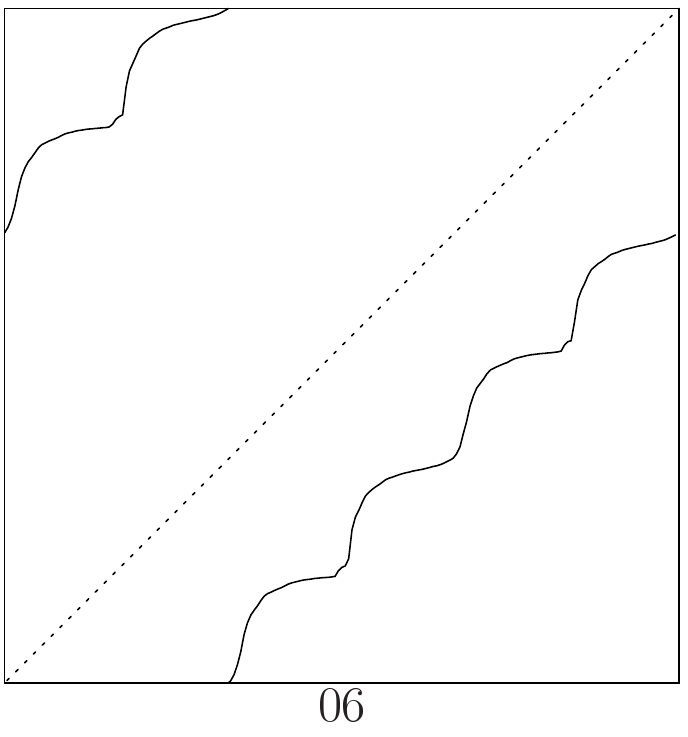} &
\includegraphics[width=\smallpic]{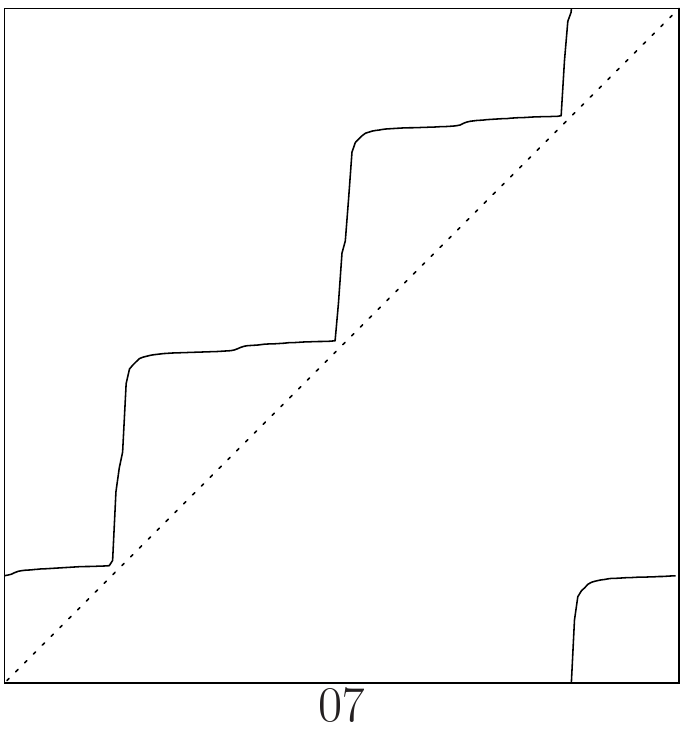} &
\includegraphics[width=\smallpic]{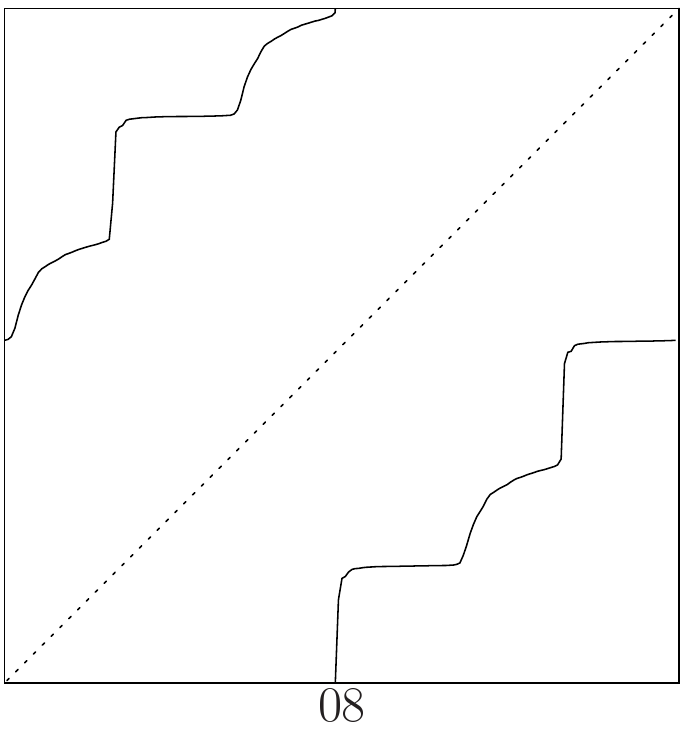} \\
\includegraphics[width=\smallpic]{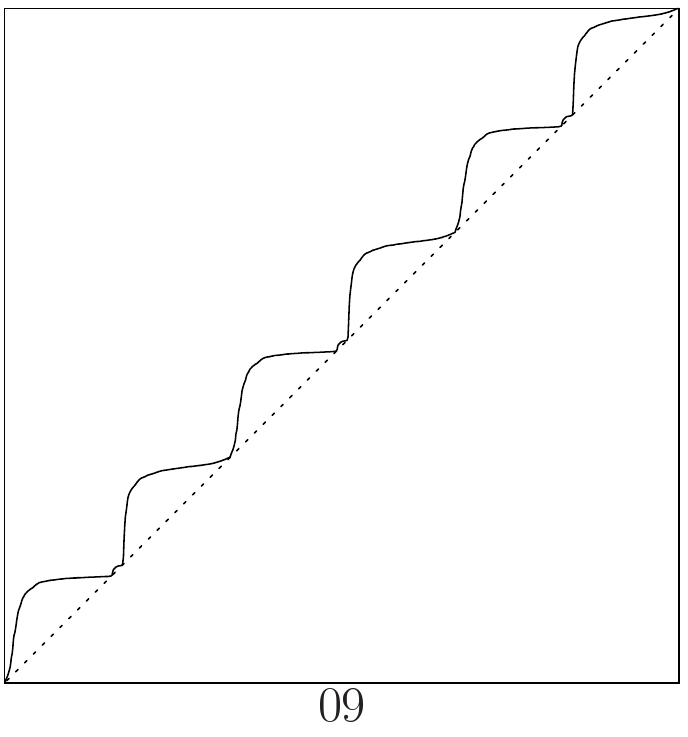} &
\includegraphics[width=\smallpic]{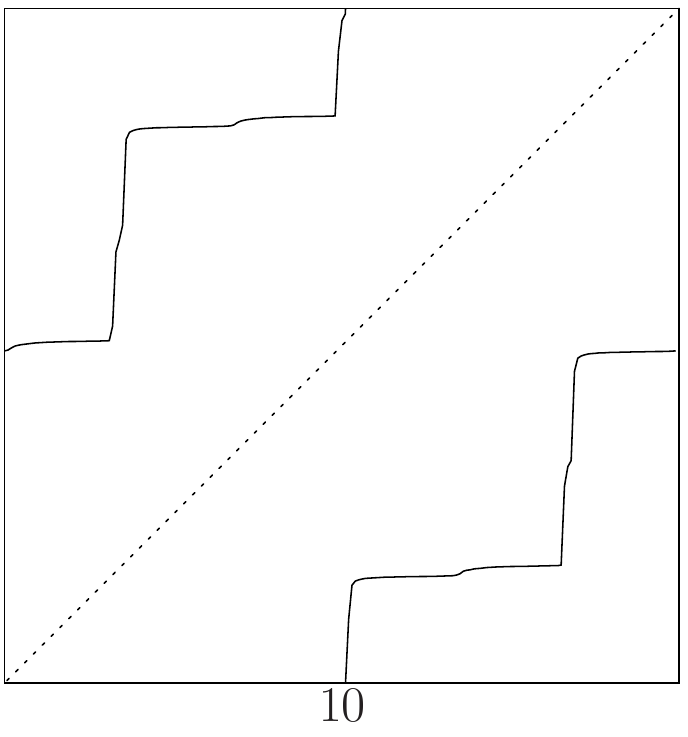} &
\includegraphics[width=\smallpic]{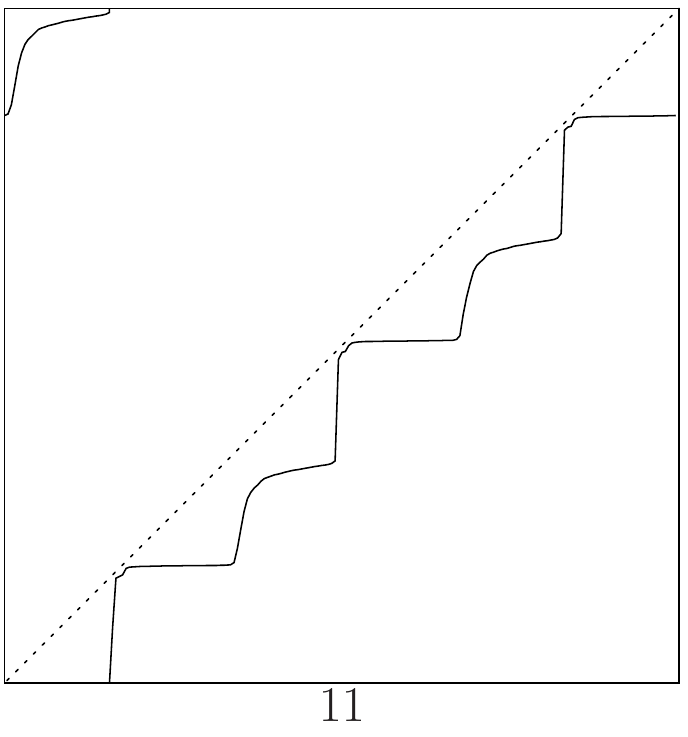} &
\includegraphics[width=\smallpic]{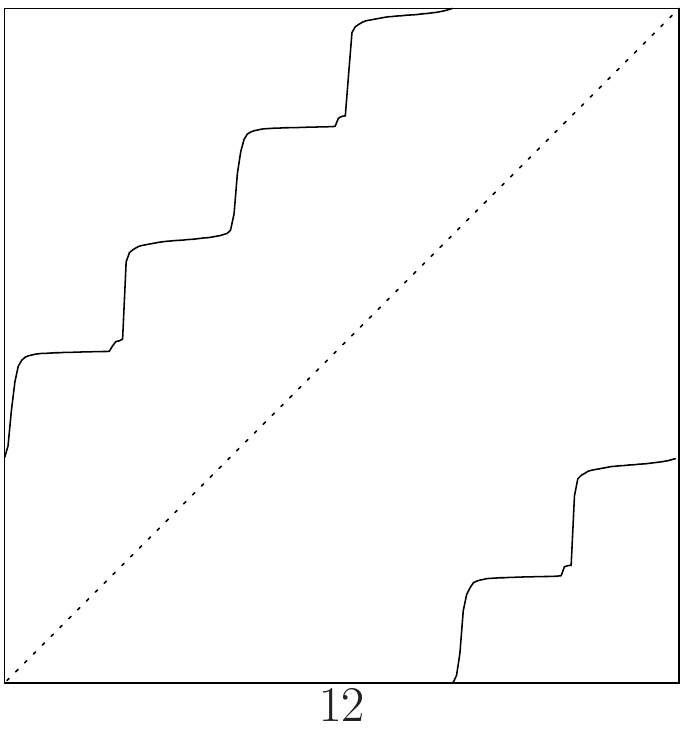} \\
\includegraphics[width=\smallpic]{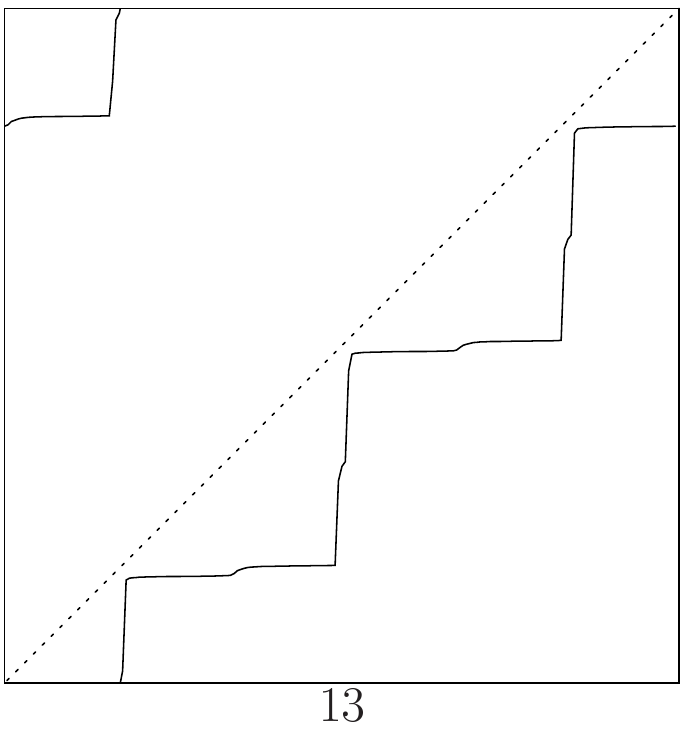} &
\includegraphics[width=\smallpic]{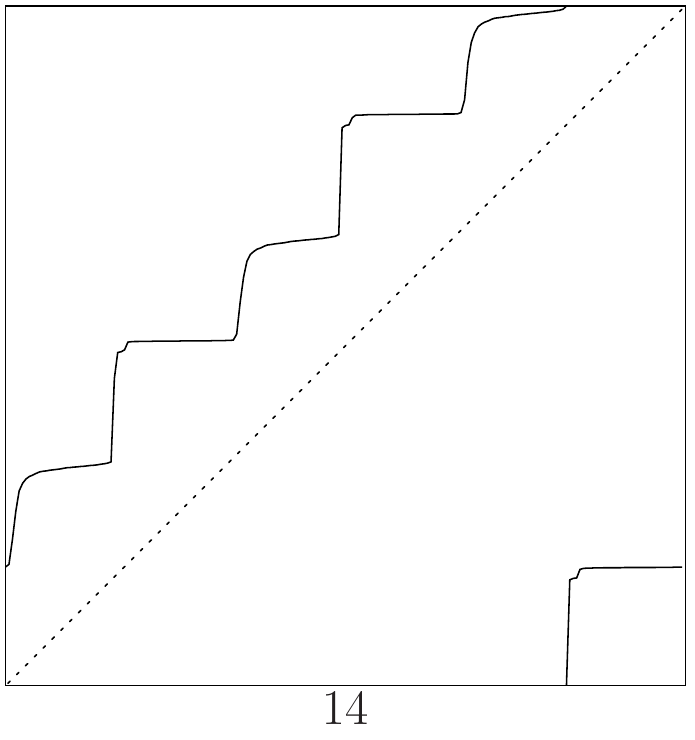} &
\includegraphics[width=\smallpic]{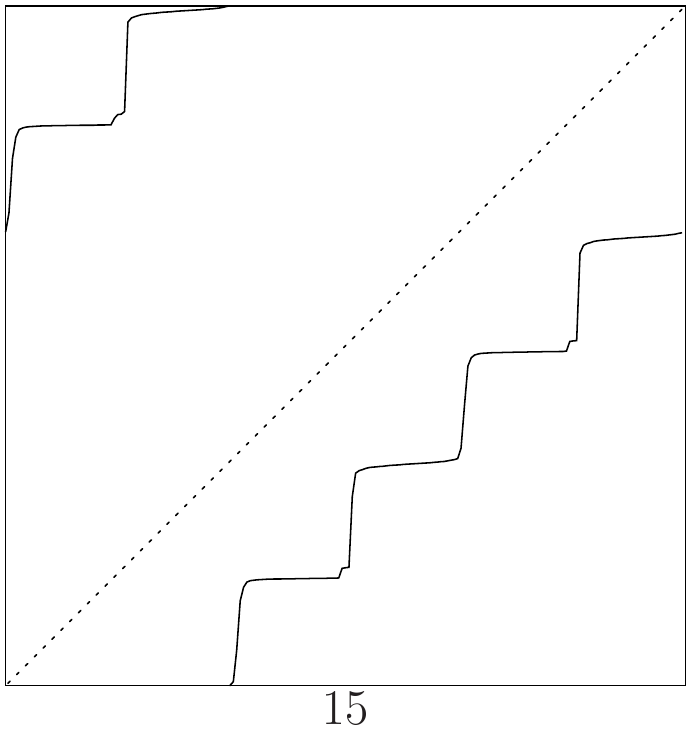} &
\includegraphics[width=\smallpic]{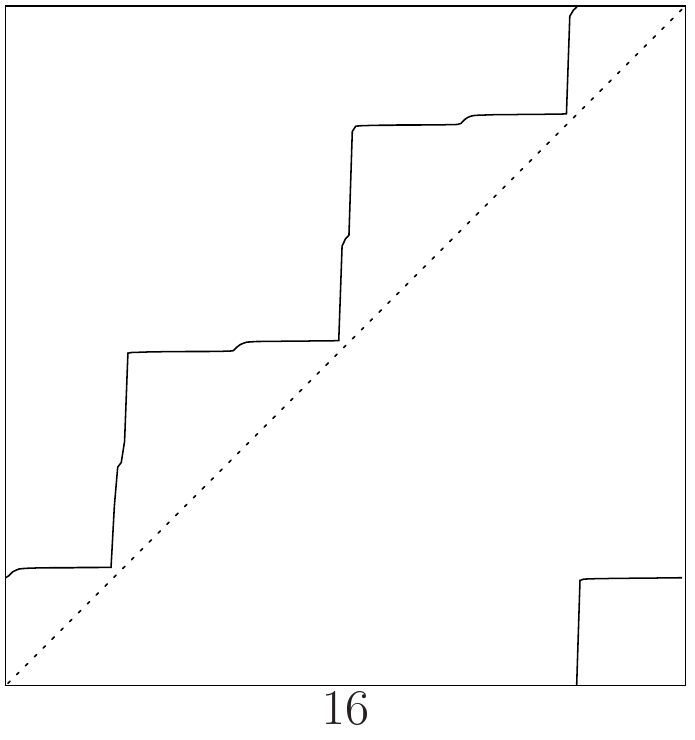} \\
\includegraphics[width=\smallpic]{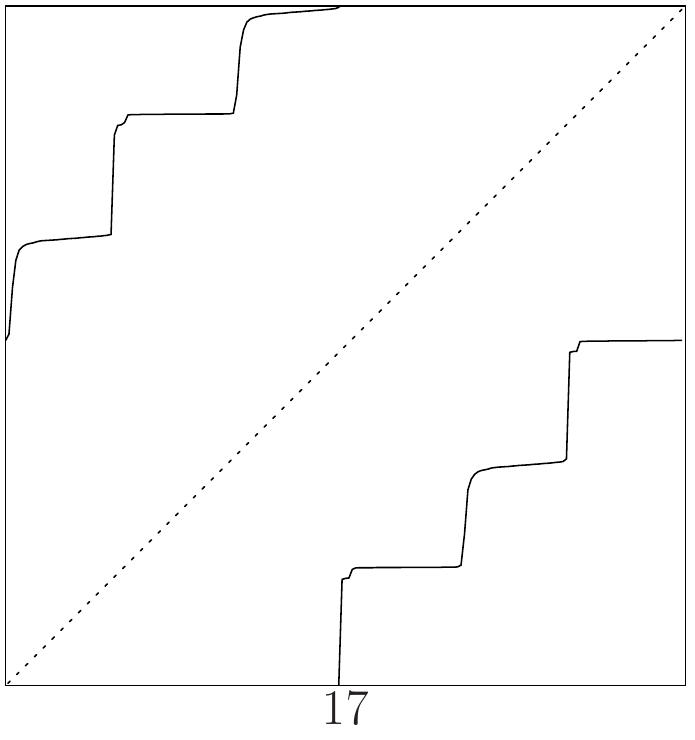} &
\includegraphics[width=\smallpic]{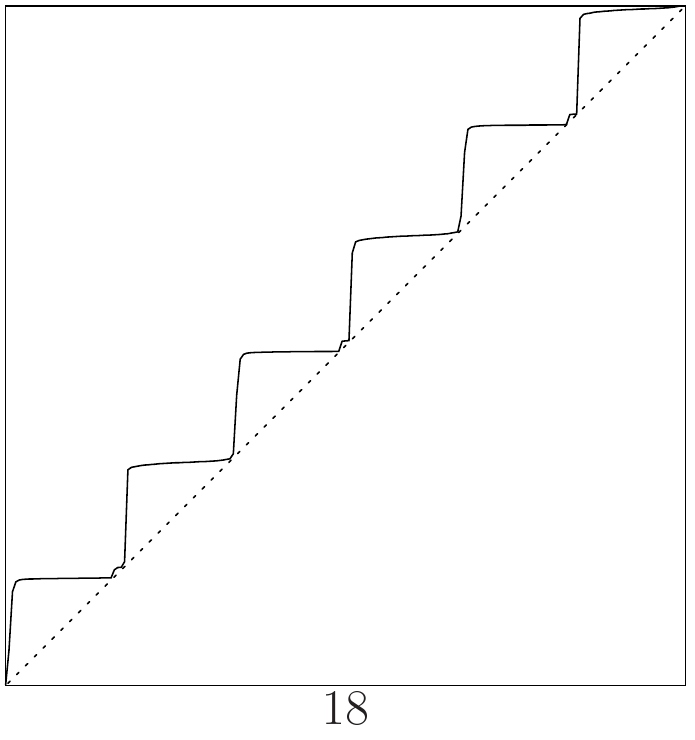} &
\includegraphics[width=\smallpic]{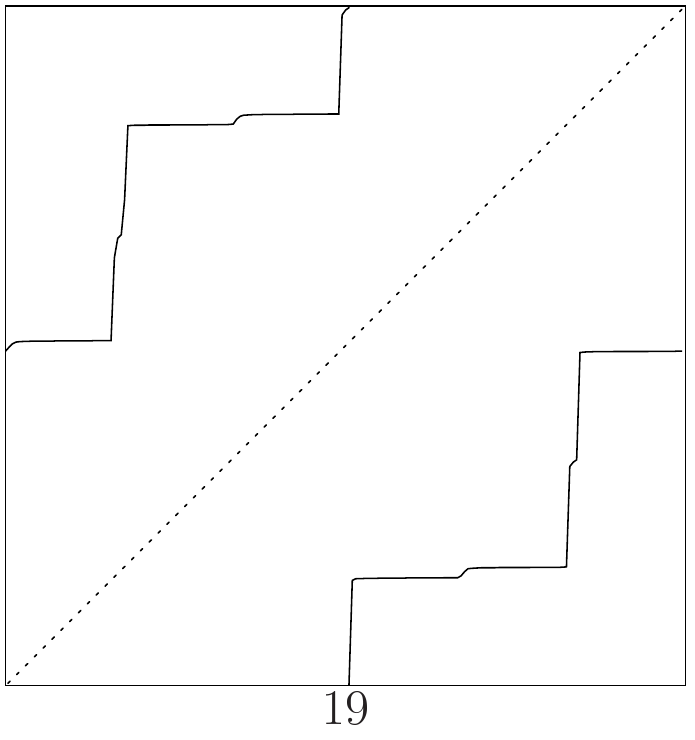} &
\includegraphics[width=\smallpic]{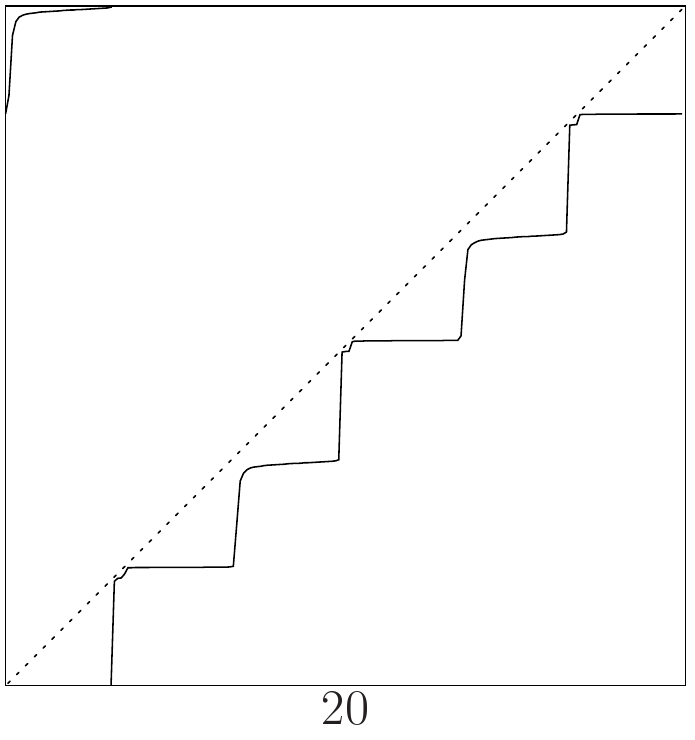} \\
\includegraphics[width=\smallpic]{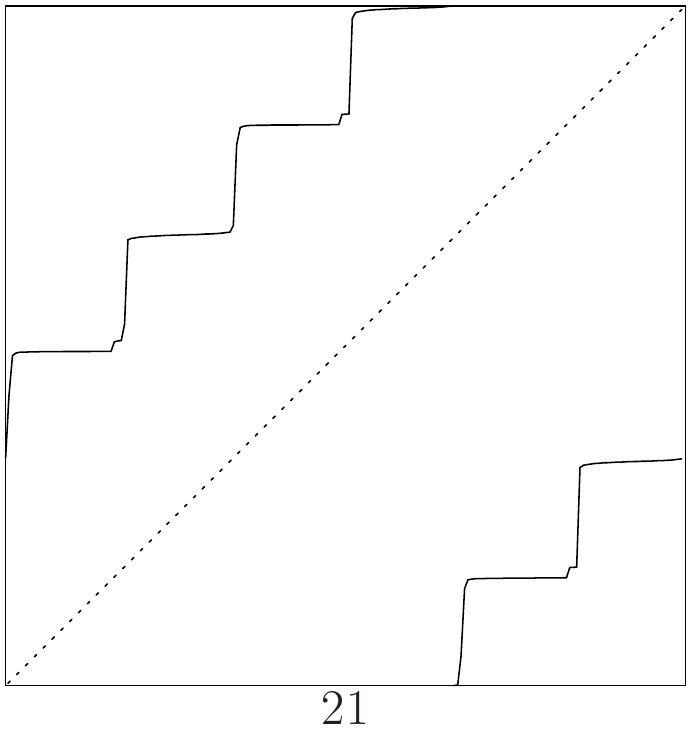} &
\includegraphics[width=\smallpic]{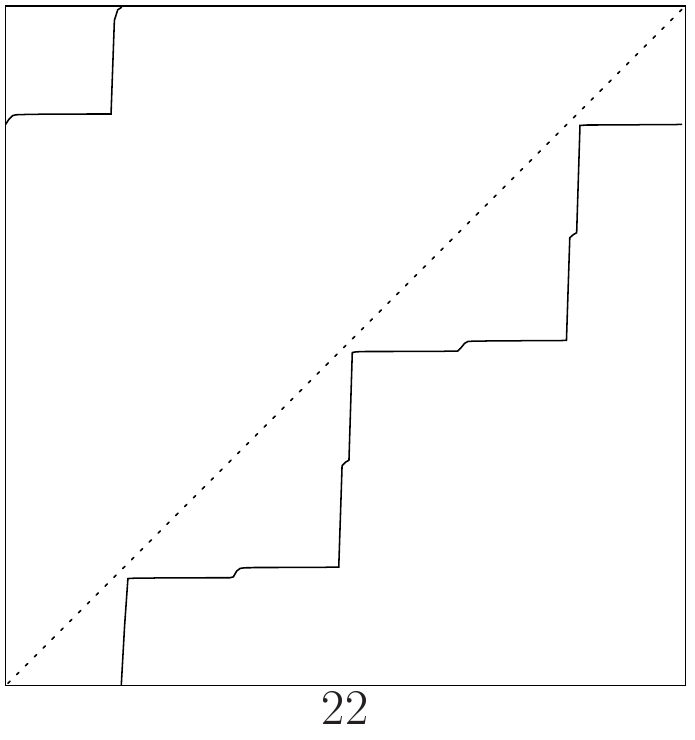} &
\includegraphics[width=\smallpic]{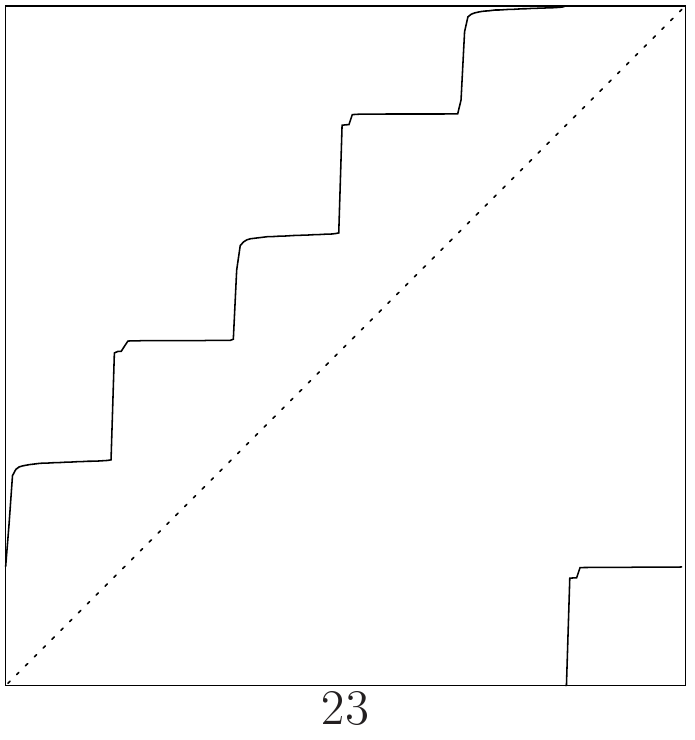} &
\includegraphics[width=\smallpic]{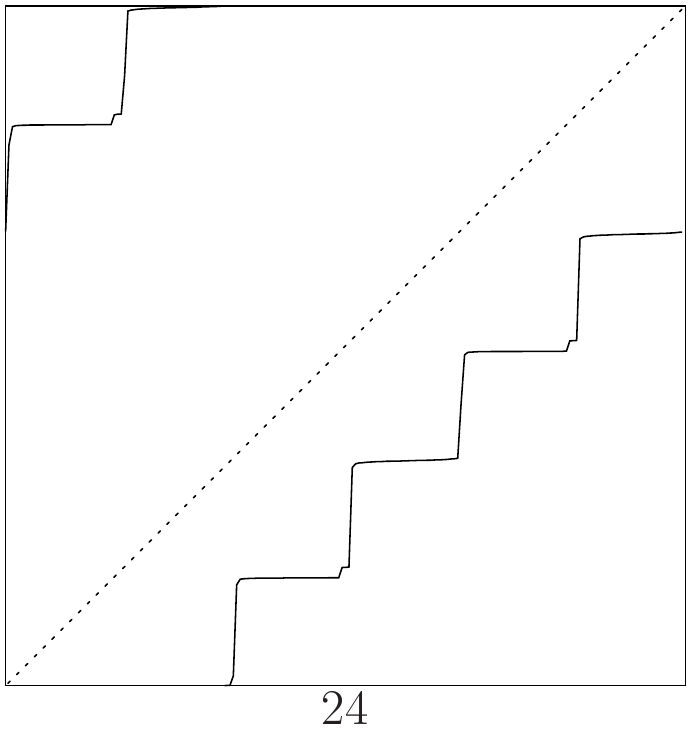} \\
\includegraphics[width=\smallpic]{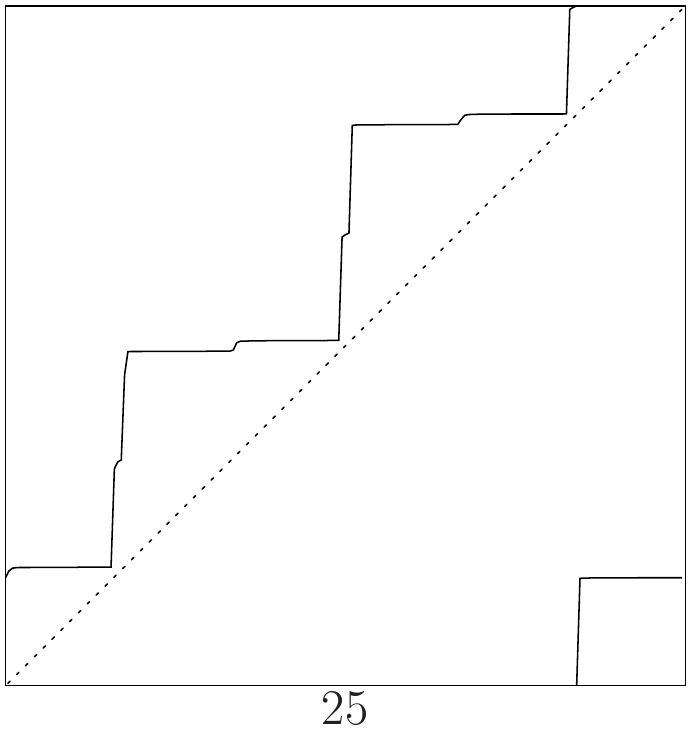} &
\includegraphics[width=\smallpic]{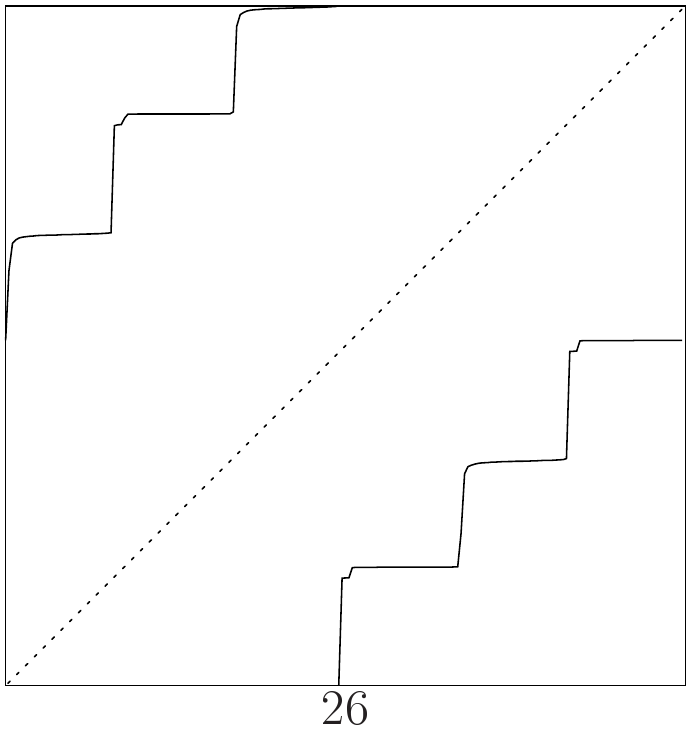} &
\includegraphics[width=\smallpic]{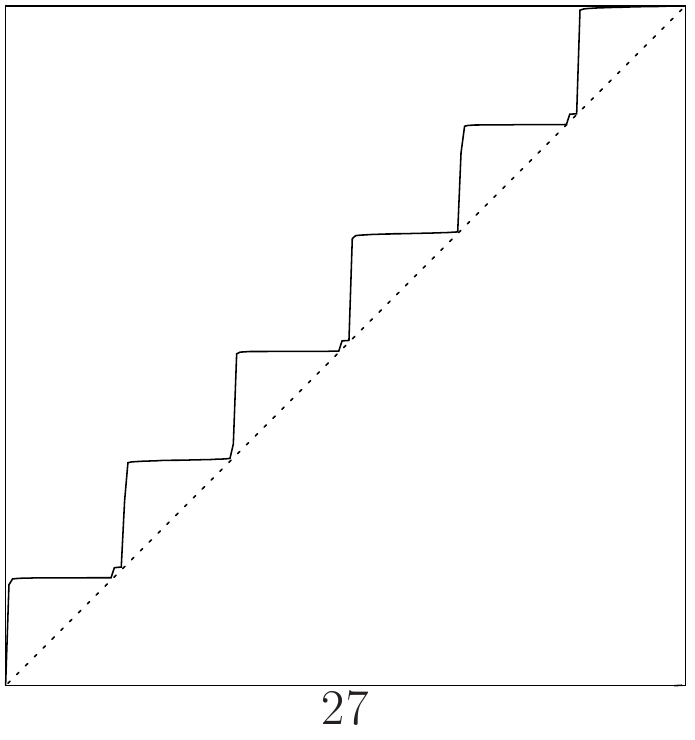} &
\includegraphics[width=\smallpic]{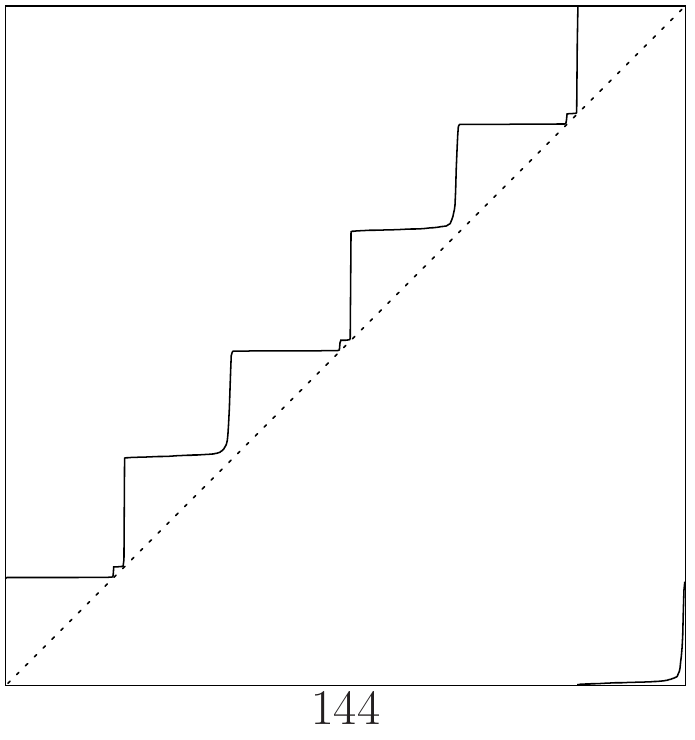}
\end{tabular}
\end{center}
\caption{
        Plots of $\sigma^n:\Sph\longrightarrow\Sph$. In plot 09 the function $\sigma^9$ seems to touch the 
        diagonal (see also Figure \ref{fig:sigma9} for an enlarged plot). 
        Accordingly $\sigma^i$ touches the diagonal in 
        plot $i$ for $i \in \{9,18,27,\dots\}$. Note that in plot 144 the numerical errors have added up so that 
        $\sigma^{144}$ no longer touches the diagonal.
        } \label{fig:sigmatable}
\end{figure}

We believe that $b_9$ is indeed a cycle (see Figure \ref{fig:billiard}):
\begin{conjecture}[Existence of nine-cycle] \label{con:ninebilliard}
Let $\trefoil: \Sph \longrightarrow \RR^3$ be the ideal trefoil, 
parameterized with constant speed such that $\trefoil(0)$ is the outer point of the trefoil on a symmetry axis.
Then $b_9=(s_0,\ldots,s_8)$ with $s_i:=\sigma^i(0)$ is a nine-cycle. Numerics suggest that $\trefoil$ passes
from 0 to 1 through $s_i$ in the sequence: $s_0, s_7,  s_5, s_3, s_1, s_8, s_6, s_4, s_2$.
\end{conjecture}

\begin{figure}
\begin{center}
\includegraphics[width=\textwidth]{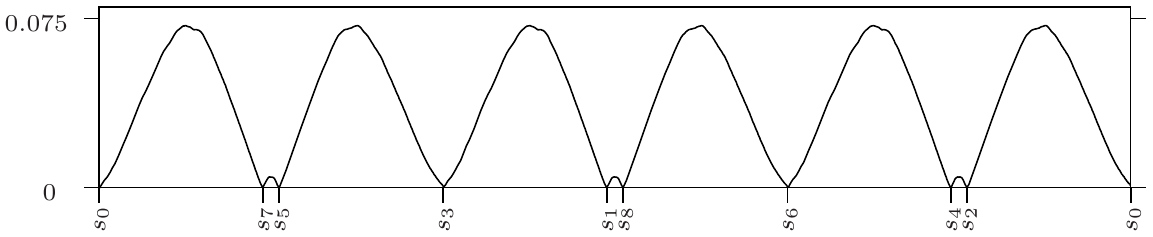} 
\caption{The plot shows graph number $9$ from Figure \ref{fig:sigmatable}
rotated by $45^\circ$. It seems to touch the diagonal 9 times
        in the points $s_0=0$, $s_7=0.159$, $s_5=0.175$,
        $s_3=0.334 \approx 1/3$, $s_1= 0.492$, $s_8=0.508$,
        $s_6=0.667 \approx 2/3$, $s_4=0.826$, $s_2=0.841$.
        This indicates the existence of a nine-cycle $b_9=(0,\sigma^1(0),\ldots,\sigma^8(0))$.
        } \label{fig:sigma9}
\end{center}
\end{figure}

Note that $b_9$ partitions the trefoil in 9 parts (see Figure \ref{fig:billiard}): Three curves
\begin{eqnarray*}
  \beta_1 &:=& \trefoil|_{[s_0,s_7]}, \\
  \beta_2 &:=& \trefoil|_{[s_6,s_4]}, \\
  \beta_3 &:=& \trefoil|_{[s_3,s_1]}, \\
  \text{or} \; \beta_i &:=& \trefoil|_{[s_{6(i-1)},s_{6(i-1)-2}]} \; \text{with} \; s_k=s_{k+9},  \\
\end{eqnarray*}
which are congruent by $120^\circ$ rotations around the $z$-axis. 
Another three curves
\begin{eqnarray*}
  \tilde{\beta_1} &:=& \trefoil|_{[s_2,s_0]}, \\
  \tilde{\beta_2} &:=& \trefoil|_{[s_8,s_6]}, \\
  \tilde{\beta_3} &:=& \trefoil|_{[s_5,s_3]}, \\
  \text{or} \; \tilde{\beta_i} &:=& \trefoil|_{[s_{6(i-1)+2},s_{6(i-1)}]} \; \text{with} \; s_k=s_{k+9},  \\
\end{eqnarray*}
which are again congruent by $120^\circ$ rotations and with each $\tilde{\beta_i}$ congruent to $\beta_i$ by a 
$180^\circ$ rotation.
And finally three curves
\begin{eqnarray*}
  \alpha_1 &:=& \trefoil|_{[s_1,s_8]}, \\
  \alpha_2 &:=& \trefoil|_{[s_7,s_5]}, \\
  \alpha_3 &:=& \trefoil|_{[s_4,s_2]}, \\
  \text{or} \; \alpha_i &:=& \trefoil|_{[s_{6(i-1)+1},s_{6(i-1)-1}]} \; \text{with} \; s_k=s_{k+9}, \\
\end{eqnarray*}
which are congruent by rotations of $120^\circ$ and self congruent by a rotation of $180^\circ$.

Because $b_9$ is a cycle, each piece of the curve gets mapped one-to-one to another piece of the curve.
\begin{lemma}[Piece to piece] \label{lem:piece2piece}
  Assume that the ideal trefoil admits a contact function $\sigma$ as in Definition 
  \ref{def:sigma} and 
  Conjectures \ref{con:sym}, \ref{con:ninebilliard} about symmetry and the
  existence of a nine cycle $b_9=(s_0,\ldots,s_8)$ hold. Then $\sigma$ maps each parameter 
  interval $[s_i,s_j]$ to $[s_{i+1},s_{j+1}]$. 
   In particular: Following the contact in $\sigma$ direction
  we get the sequence
  $\alpha_1\to\tilde{\beta_1}\to\beta_3\to \alpha_3\to\tilde{\beta_3}\to\beta_2\to 
   \alpha_2\to\tilde{\beta_2}\to\beta_1 (\to \alpha_1)$. 
  Each piece is in one-to-one contact with the next in the sequence 
  (see also Figure \ref{fig:connect}).
\end{lemma}
\begin{proof} By definition $s_i$ is mapped to $s_{i+1}$ and by Definition
\ref{def:sigma} the contact function $\sigma$ is continuous and orientation
preserving so the interval $[s_i,s_j]$ gets mapped to
  $[\sigma(s_i),\sigma(s_j)] = [s_{i+1},s_{j+1}]$.
\end{proof}

One further remark about the plots in Figure \ref{fig:sigmatable}. If $s_i \in \Sph$ is a solution of $\sigma^9(s_i)=s_i$ then by Lemma 
\ref{lem:sigma_sym} the parameter $r=s_i+k/3$ is also a solution of $\sigma^9(r)=\sigma^9(s_i+k/3)=s_i+k/3=r$ for $k \in \{0,1,2\}$. Since there are
presumably only nine solutions $s_i$, they happen to fall in three classes represented by $s_0, s_1, s_2$ and with $s_{i+3k} = s_i+k/3$ the remaining six
are defined. Consequently we find that $\sigma^3(s_i)=s_{i+3}=s_i+1/3$, i.e. there are nine solutions 
of $\sigma^3(s)=s+1/3$ which can be seen in plot number 3 of Figure \ref{fig:sigmatable}. 
Similarly, there are nine solutions of $\sigma^6(s)=s+2/3$ in plot number 6 and so on.

We now briefly discuss the relationship between particular points in
the curvature plot and the closed-cycle points $(s_0,\cdots,s_8)$,
i.e. the partitioning introduced above. In Figure \ref{fig:curvature}
we show the curvature plot scaled by the thickness $\Delta$ on the interval $[0,1/3] = [s0,s3]$. 
Since curvature is confined in $[0,1/\Delta]$ for thick knots this always gives a comparable graph.
Due to the $3$-symmetry, the plots on the intervals $[1/3,2/3]$ and $[2/3,1]$ are identical.
The $180$-degree rotation symmetry shows up in the plot as a symmetry
around $(s_7+s_5)/2$, the center of a self-congruent piece $\alpha_2$.
The curvature profile is close to constant $.5$ on the major part 
$\beta_1$ and $\tilde{\beta}_3$. 
A significant change occurs at the transition points
between $\alpha_i$ and $\beta_i$, where it reaches its maximum
at the junction points $s_7$ and $s_5$, where curvature is believed to
be active\cite{bookchapter, BPP08}. 
The spikes of our computation do not achieve the maximal value, and there is a 
local maximum at the center of an $\alpha_i$ piece. We believe these deviations from
earlier observations are numerical artefacts due to the Fourier 
representation used to compute this trefoil.\footnote{In fact the curvature function needs not even to
 converge, as one approaches an ideal shape \cite[Section 2.5]{Ge10}}.
The alignment of the closed-cycle points and the points where
curvature seems active only enforces that all the numerical pieces
fit together nicely, which is a good indication, that these are not
numerical artefacts.

\begin{figure}
\begin{center}
\includegraphics[width=0.6\textwidth]{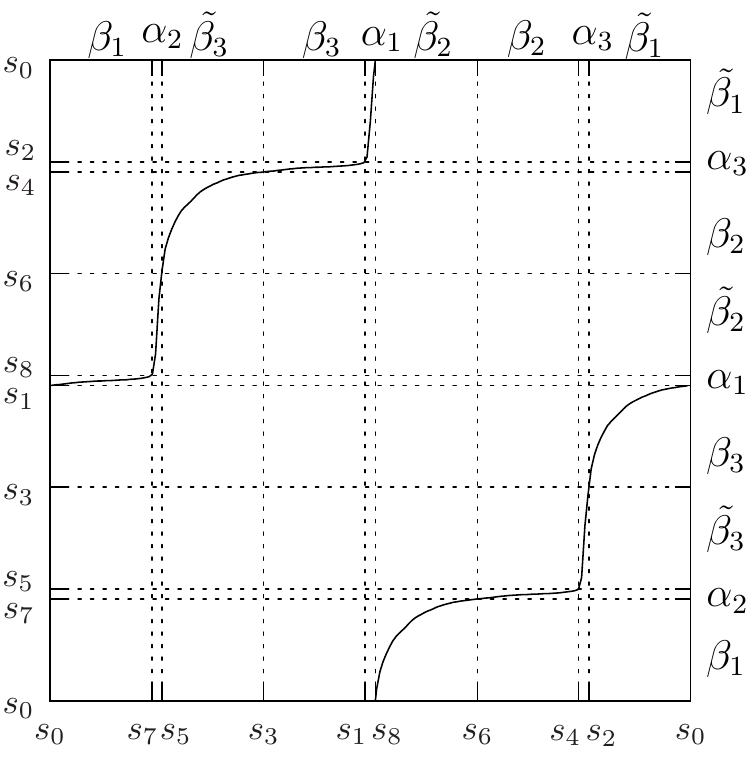} 
\caption{Plotting the $\sigma$ function with a grid of the partition points $s_i$ we can read off which piece of 
         the curve is in contact (in the $\sigma$ direction) with which other piece. For example starting at the
         top with $\alpha_1$ we see that $\sigma$ maps its parameter-interval to the parameter-interval 
         of $\tilde{\beta_1}$ on the right, which itself gets mapped from top $\tilde{\beta_1}$ to $\beta_3$
         on the right and so on (see Lemma \ref{lem:piece2piece}).
        } \label{fig:connect}
\end{center}
\end{figure}

\begin{figure} 
  \begin{center}
      \includegraphics[width=\textwidth]{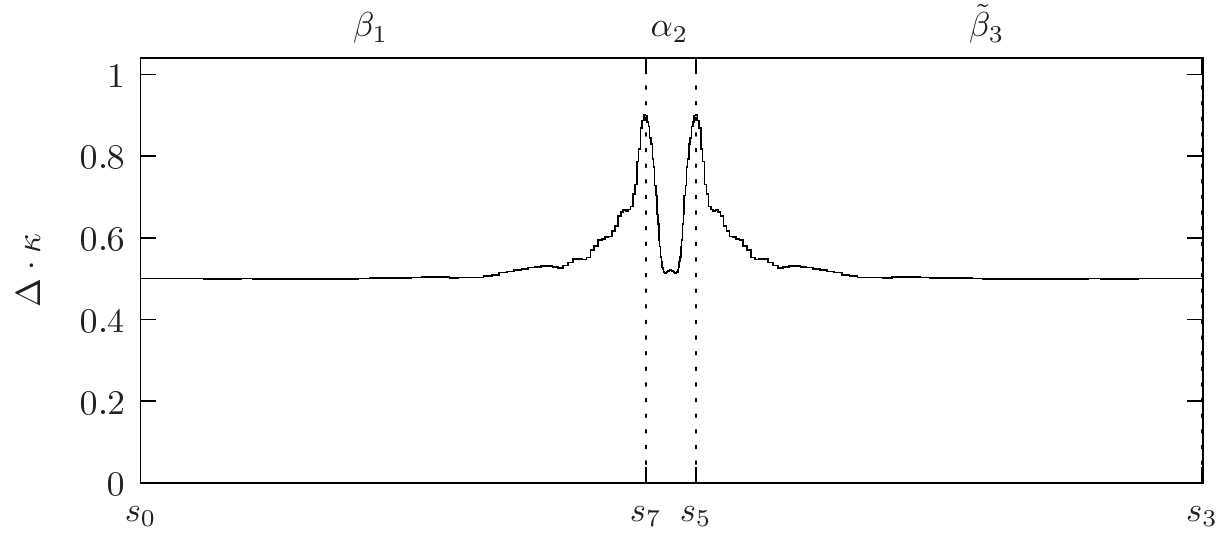}
  \end{center} 
 \caption{The curvature $\kappa$ of $\trefoil$ is confined to $[0,1/\Delta]$. Consequently the graph shows $\kappa\cdot\Delta$.
          Since $\kappa$ is three-periodic, for greater detail we show only one third of the interval. 
          The maximal curvature is attained at $s_7$ and $s_5$.
          } 
 \label{fig:curvature}
\end{figure}

\begin{figure}
\begin{center}
\includegraphics[width=\textwidth]{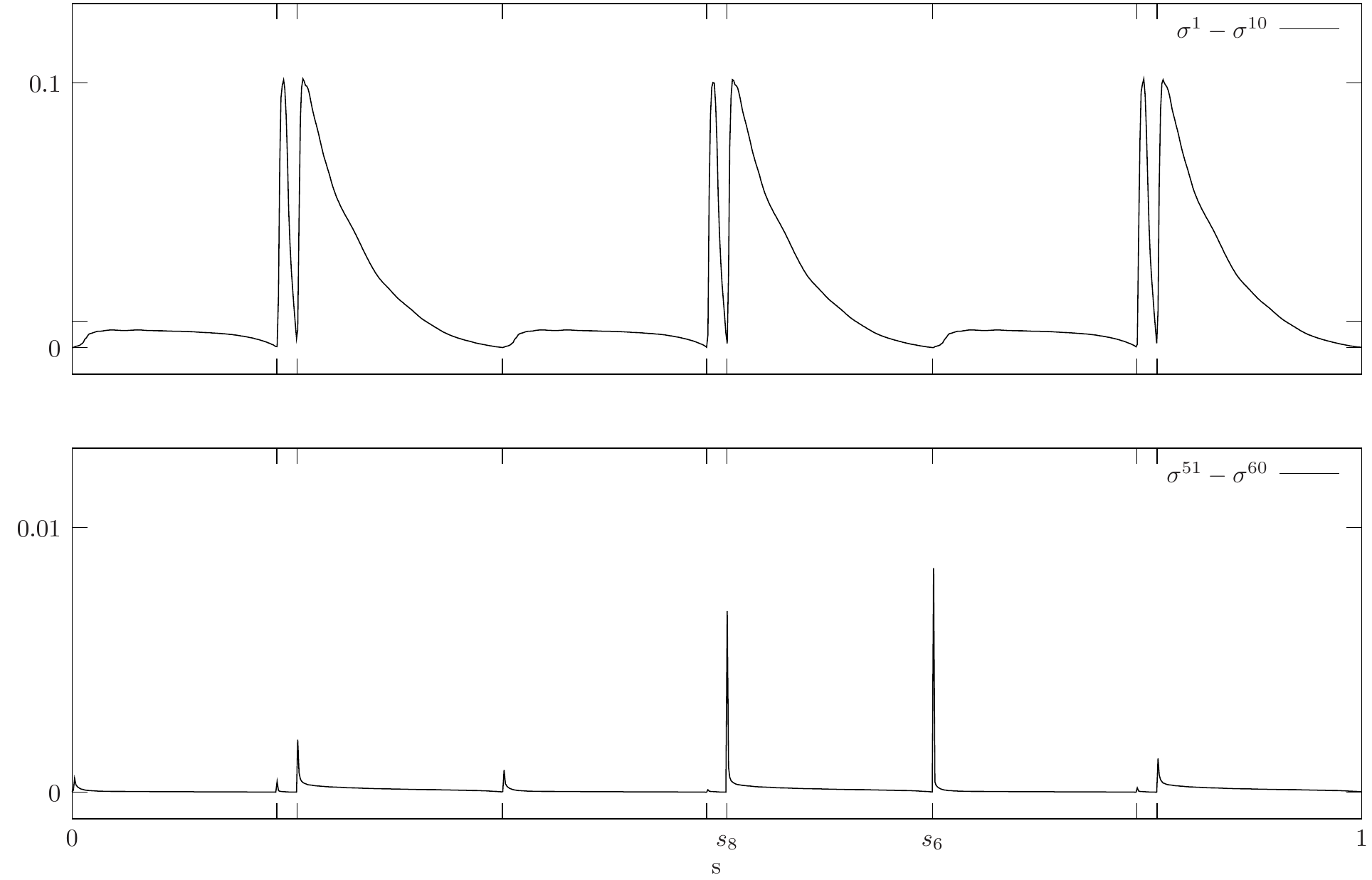} 
\caption{
        Since numerical experiments find that $\sigma^{i+9}-\sigma^i$ converges point-wise to $0$ 
        for $i\to\infty$ we conjecture that 
        the cycle $b_9$ acts as an attractor, i.e. $(\sigma^{9i}(s), \sigma^{9i+1}(s), \ldots,\sigma^{9i+8}(s))$ converges
        to $b_9$ up to a cyclic permutation. Notice that the convergence is only point-wise
        and cannot be uniform since $\sigma$ is continuous; with enough samples we would
        see large spikes after each $s_i$ as behind $s_6$ and $s_8$ above.
        } \label{fig:attractor}
\end{center}
\end{figure}

\begin{figure}
\begin{center}
\includegraphics[width=\textwidth]{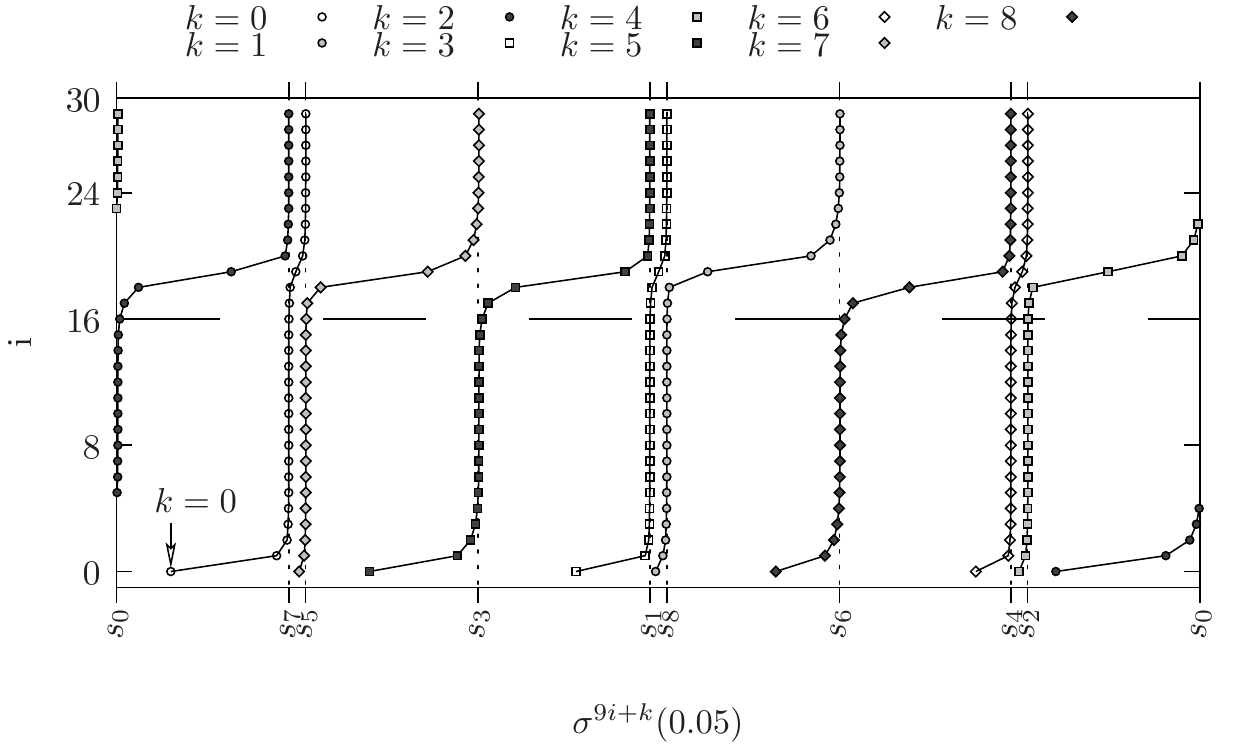} 
\caption{The cycle $b_9$ seems to be an attractor: starting at the arbitrary point $t=0.05$ the sequence $\sigma^{9i}(t)$ 
         seems to converge to $s_7$ for growing $i$. After $i>16$ iterations, presumably, errors in the 
         approximation and the numerics add up and $\sigma^{9\cdot 18}(t)$ is in the next 
         interval $(s_7,s_5)$. Starting from $\sigma^{k}(t)$ in other intervals shows a similar behavior.
        }
         \label{fig:attractorlines}
\end{center}
\end{figure}

Taking a second look at Figure \ref{fig:sigmatable} it looks like $\sigma^n$ is approaching a 
step function as $n$ increases. What are the accumulation points of the sequence $\{\sigma^i(t)\}_i$ as a function of 
$t\in\Sph$?  Looking at $(\sigma^{i}(t), \sigma^{i+1}(t), \ldots, \sigma^{i+8}(t))$ for arbitrary $t \in \Sph$
it seems to converge to $b_9$ up to a cyclic permutation for $i$ large enough, i.e. $b_9$ contains the
accumulation-points of the above sequence. Figure \ref{fig:attractor} shows some numeric values
of $\sigma^{i+9}-\sigma^i$ which seems to converge 
point-wise 
to $0$ for $i\to\infty$. An arbitrary point $t$ between neighboring points $s_l$ and $s_r$ gets by each application of
$\sigma^9$ repelled from the left by $s_l$ and attracted to the right by $s_r$ (see 
Figure \ref{fig:attractorlines}).\footnote{We would like to thank E. Starostin for encouraging 
us to take a closer look at this issue.} 
Note that the attactor has a direction that is induced by the chirality of the trefoil
(left or right-handed) and the choice of the contact function $\sigma$ made in Definition \ref{def:sigma}.

\begin{conjecture}[Attractor]\label{con:attractor}
Let $b_n \in (\Sph)^n$ be a cycle. We call $b_n$ an attractor if
for any $t \in \Sph$ and fixed $k \in \{0,\ldots,n-1\}$ the $n$-tuple 
$(\sigma^i(t), \sigma^{i+1}(t), \ldots,\sigma^{i+n-1}(t))$ 
converges to a cyclic permutation of $b_n$ for
$j\to\infty$ with $i=nj+k$.
The $b_9$ cycle of Conjecture \ref{con:ninebilliard} is an attractor.
\end{conjecture}

The existence of an attractor rules out the existence of other cycles:

\begin{lemma}\label{lem:one_attractor}
Assume $\trefoil$ has a contact function $\sigma$ as in Definition \ref{def:sigma} and 
let $b_n \in (\Sph)^n$ be a $n$-cycle as in Conjecture \ref{con:ninebilliard}.
 Then $b_n$ is an attractor in the sense of Conjecture \ref{con:attractor}
iff there is no other $n$-cycle on $\trefoil$.
\end{lemma}
\begin{proof}
Assume that $c_n=(t_0,\sigma(t_0),\ldots,\sigma^{n-1}(t_0))$ is a $n$-cycle different from $b_n$. 
Then $\sigma^{n}(t_0)=t_0$ and $b_n$ cannot be an attractor.

To prove the converse, let $b_n=(s_0,\ldots,s_{n-1})$ be an $n$-cycle and
consider $\sigma^n$ as a continuous, injective and orientation preserving 
map from $[s_k,s_l] \subset \RR$ to itself, where $l$ and $k$ 
are such that $s_l$ and $s_k$ are neighboring, i.e. $s_i \not\in (s_k,s_l)$ for all $i$.
The cycle $b_n$ is an attractor iff for all $x \in (s_k,s_l]$ the sequence
$x_i:=\sigma^{ni}(x)$ converges to $s_l$ as $i \to \infty$.
Since $\sigma$ is orientation preserving we have $x_i \le x_{i+1}$, i.e. the sequence is monotone.
Assume that $b_n$ is not an attractor, i.e. for some $x$ the sequence 
$\{\sigma^{ni}(x)\}_i$ is bounded away from $s_l$, then 
it must converge to some smaller value $c < s_l$. By continuity of $\sigma^n$ it follows that $c$ is a fixed point. 
Therefore $c_n:=(c,\sigma^1(c),\sigma^2(c),\ldots,\sigma^n(c))$ is a closed $n$-cycle different from $b_n$.
\end{proof}

As mentioned above, by concatenation, a minimal cycle gives rise to a series of larger, non-minimal cycles:
For an $n$-cycle $a_n=(x_0,\ldots,x_{n-1})$ we define an $nk$-cycle
\[
  a^k_{n} := (\underbrace{x_0,\ldots,x_{n-1},\ldots,x_0,\ldots,x_{n-1}}_{k \text{ times}}).
\]

\begin{proposition}\label{prop:one_cycle}
Assume $\trefoil$ has a contact function $\sigma$ as in Definition \ref{def:sigma} and 
let $b_n \in (\Sph)^n$ be a minimal $n$-cycle as in Conjecture \ref{con:ninebilliard} and an 
attractor in the sense of Conjecture \ref{con:attractor}. 
Then $b_n$ is the only minimal cycle and all other cycles are multiples of $b_n$.
\end{proposition}
\begin{proof}
Let $c_m$ be a $m$-cycle different from $b_n$.
We claim $c_m=b^k_n$ and $m=kn$ for some $k \in \NN$.

If $b_n$ is an attractor, then $b^i_n$ is also an attractor for $i \in \NN$. 
Let $g$ be the greatest common divisor of $n$ and $m$, so $l = mn/g$ is their least common multiple.
Then $b_n^{m/g}$ is an attractor and an $l$-cycle, but $c_m^{n/g}$ is an $l$-cycle as well and Lemma \ref{lem:one_attractor} 
implies $b^{m/g}_n=c^{n/g}_m$. Since $b_n$ was minimal, this is only possible if $c_m=b^k_n$ for some $k\in\NN$.
\end{proof}

Lemma \ref{lem:one_attractor} and Proposition \ref{prop:one_cycle} fit well with our numerical observation. We find only one possible cycle and it seems to
be an attractor.

\section{Conclusion}

We have presented numerical and esthetical compelling evidence for the existence of 
a closed nine-cycle in the contact chords of the
ideal trefoil knot. Enforcing symmetry based on a Fourier representation
turned out to be essential to observe this feature. The cycle leads to a
partitioning of the trefoil. Only two segments of the curve have to
be considered, the remaining parts of the trefoil can be reconstructed
by symmetry. For other contact chord paths, after enough iterations,
it seems that they eventually converge to the nine-cycle. So the
closed cycle acts as an attractor for all other billiards.

Preliminary numeric experiments by E. Starostin suggest closed cycles 
in ideal shapes of knots with a higher number of crossings as well. 
The interesting cases remain however inconclusive, 
since these knot shapes are believed to be much less ideal than the trefoil. 
Closed cycles in these knots
might then also suggest a natural partitioning of the curves, therefore
improving the understanding of these knots.

In the $\Sph^3$ setting \cite{Ge10} suggested a candidate trefoil for
ideality to the problem of maximizing thickness.
Each point on the $\Sph^3$ trefoil is in contact with two other
points on the curve. Following the contact great-arcs (in $\Sph^3$) 
five times forms a circle, i.e. a 5-cycle.

The numerical computations suggest at least two new challenges: First, can we get new insights about the ideal trefoil assuming the existence of a nine-cycle? 
And second, can we prove, under some reasonable hypothesis, that the 
ideal trefoil or even every ideal knot has closed contact cycles?

\section{Acknowledgements} 
Research supported by the Swiss National Science Foundation SNSF No. 117898
and SNSF No. 116740.
We would like to thank  E. Starostin and J.H. Maddocks for
interesting discussions and helpful comments.

\end{document}